\documentclass[11pt,reqno]{amsart}
\usepackage[margin=1.4in]{geometry}
\usepackage{derivative}
\usepackage{amssymb} 

\usepackage{xcolor}
\usepackage{lmodern}
\usepackage{hyperref, amsmath, amssymb, mathtools}
\usepackage{verbatim}
\usepackage{amsthm}
\usepackage{amsfonts}
\usepackage{amsmath}
\usepackage{mathrsfs}  
\usepackage{comment}
\usepackage{chngcntr}
\usepackage{ctable}
\usepackage{url}
\usepackage{hyperref}
\usepackage[numbers]{natbib}
\newtheorem{thm}{Theorem}
\newtheorem*{thm*}{Theorem}
\newtheorem{lemma}[thm]{Lemma}
\newtheorem{corollary}[thm]{Corollary}
\newtheorem{conjecture}[thm]{Conjecture}
\newtheorem{prop}[thm]{Proposition}
\newtheorem*{prop*}{Proposition}
\theoremstyle{definition}

\newtheorem{defn}{Definition}

\numberwithin{thm}{section}

\newcommand{\Z}{\mathbb{Z}}

\newcommand{\R}{\mathbb{R}}
\newcommand{\Q}{\mathbb{Q}}

\title{The Spectrum of $\Q$-isotropic binary quadratic forms}
\author[Giorgos Kotsovolis]{Giorgos Kotsovolis}
\address{Department of Mathematics, Princeton University, Princeton, NJ 08540}
\email{gk13@princeton.edu}
\date{\today}

\begin{document}

\maketitle
\begin{abstract}
    We give a complete list of the points in the spectrum $$\mathcal{Z}=\{\inf_{(x,y)\in\Lambda,xy\neq0}{\left\vert xy\right\vert},\,\text{$\Lambda$ is a unimodular rational lattice of $\R^2$}\}$$ above $\frac{1}{3}.$ We further show that the set of limit points of $\mathcal{Z}$ with values larger than $\frac{1}{3},$ is equal to the set $\{\frac{2m}{\sqrt{9m^2-4}+3m},\text{ where $m$ is a Markoff number}\}$. 
\end{abstract}
\tableofcontents
\section{Introduction}
For a lattice $\Lambda\subset\R^2$ of covolume 1, we define $$\lambda_1(\Lambda)=\inf_{(x,y)\in\Lambda\backslash\underline{0}}{\left\vert xy\right\vert}.$$ The set of values that $\lambda_1$ assumes as $\Lambda$ varies through the unimodular lattices of $\R^2$ is the well studied Markoff spectrum $\mathcal{M}$ and could be interpreted (after some renormalization) as the set of first eigenvalues, when the operator $D=\pdv{}{x}\pdv{}{y}$ acts on the orthogonal complement of the constant functions inside $L(\R^2/\Lambda)$ (See \cite{Sarnak} for details). The points inside $\mathcal{M}\bigcap (\frac{1}{3},\infty)$, were classified by Markoff (\cite{M1},\cite{M3}), who parametrized them by the Markoff numbers; natural numbers that are solutions of the cubic equation $$x^2+y^2+z^2=3xyz.$$ When all the vectors of the lattice $\Lambda$ lie in $\sqrt{d_{\Lambda}}\cdot\Q^2$, for some positive rational constant $d_{\Lambda}$, we trivially have that $\lambda_1(\Lambda)=0.$ We shall refer to these lattices as rational. In this paper, we study the distribution of the first non-zero eigenvalue of such tori; we define $$\lambda^*_1(\Lambda)=\inf_{(x,y)\in\Lambda,\,xy\neq0}{\left\vert xy\right\vert}$$ and we wish to understand the set $$\mathcal{Z}=\{\lambda_1^*(\Lambda),\,\text{$\Lambda$ is a unimodular rational lattice of $\R^2$}\}.$$ This spectral problem first appears, in a different form as we discuss in Section \ref{Equivalent definitions}, in the work of Zaremba (\cite{Z1},\cite{Z2}), who was interested in numerical approximation of double integrals and, following Hlawka \cite{hlawka}, understood that determining the higher values of these numbers could have applications in bounding the error term of such numerical approximations. Zaremba essentially showed that  $$\mathcal{Z}=\left\{\min_{q\in\mathbb{N},0<q<b}q\left\vert\left\vert \frac{qa}{b}\right\vert\right\vert,\, \frac{a}{b}\in\Q\right\}$$ and hence we denote by $\mathcal{Z}$, both the spectrum and the function $$\mathcal{Z}:\frac{a}{b}\rightarrow\min_{q\in\mathbb{N},0<q<b}q\vert\vert \frac{qa}{b}\vert\vert,$$ for some rational $\frac{a}{b}\in\Q$ with $(a,b)=1.$ There is a clear connection between the $\mathcal{Z}$ spectrum and the notorious Zaremba conjecture, which can be interpreted as follows, in terms of the function $\mathcal{Z}$: 
\begin{conjecture}[Zaremba]
    There exists some absolute positive constant $C$, such that for every $b\in\mathbb{N},$ there exists $a\in\mathbb{N}$ such that $\mathcal{Z}(\frac{a}{b})>C.$
\end{conjecture} Zaremba considered the largest values that the function $\mathcal{Z}$ attains and proved that the largest points of the spectrum are the points $\frac{F_{n-2}}{F_n}$, where $F_n$ is the Fibonacci sequence. The rational numbers that attain the value $\frac{F_{n-2}}{F_n}$ are $\frac{F_{n-1}}{F_n}$ and $\frac{F_{n-2}}{F_n}$ . He proceeded to show that the first point that lies outside this sequence is equal to $\frac{10}{29}$, but to our knowledge his investigation of the spectrum ended there. We give a complete description of all points of $\mathcal{Z}\cap (\frac{1}{3},\infty)$. The classification of these points will be achieved in Section \ref{Nielsen moves on finite words}. For now we state the following result regarding the limit points of $\mathcal{Z}$ larger than $\frac{1}{3}:$ 

\begin{thm}\label{Limit points}
    Let $m$ denote a Markoff number. Then $$\frac{2m}{\sqrt{9m^2-4}+3m}$$ is a limit point of $\mathcal{Z}\cap (\frac{1}{3},\infty).$ Furthermore, these are all the limit points in $\mathcal{Z}$ larger than $\frac{1}{3}$.
\end{thm}
We further prove the following regarding the distribution of the limit points of $\mathcal{Z}$ below $\frac{1}{3}$:
\begin{thm}\label{Uncountably}
    For every $\epsilon>0$, there exist uncountably many limit points of $\mathcal{Z}$ in $(\frac{1}{3}-\epsilon,\frac{1}{3}).$
\end{thm}
\section{Equivalent definitions of the $\mathcal{Z}$ spectrum}\label{Equivalent definitions}

Let $p>1$ denote some natural number and $g=(g_1,g_2)$ some non-zero integer vector in $\R^2$ coprime to $p$. Zaremba's considerations regarding good bounds for the numerical computation of double integrals led him to consider integer lattices of the form $$\Lambda_{p,g}=\left\{(u,v)\in\Z^2: ug_1+vg_2\equiv 0\mod p\right\},$$ calling a vector $g$ ``good" if the quantity $p^{-1}\inf_{(u,v)\in\Lambda_{p,g},uv\neq 0}\vert uv\vert$ is large. It is clear that $\Lambda_{p,g}$ has volume $p$ and thus $\Lambda'_{p,g}=\frac{1}{\sqrt{p}}\Lambda_{p,g}$ is a unimodular rational lattice with $$\lambda_1^*(\Lambda'_{p,g})=p^{-1}\inf_{(u,v)\in\Lambda_{p,g},uv\neq 0}\vert uv\vert.$$ Our definition of the $\mathcal{Z}$-spectrum should now be clear. However, we give another equivalent definition in the language of continued fractions, which shall be fruitful for our considerations: For every unimodular rational lattice $\Lambda,$ there exists by definition some rational number $d$ such that $(c,a),(0,b)$ is a basis of the integral lattice $\frac{1}{\sqrt{d}}\Lambda.$ Since $\Lambda$ is unimodular, we must have $bc=d.$ Then $$\lambda_1^*(\Lambda)=\inf_{q,l,qc(qa+lb)\neq0}\vert d^{-1}qc(qa+lb) \vert=\inf_{q,l,qc(qa+lb)\neq0}\vert q(q\frac{a}{b}+l) \vert=\min_{q\in\mathbb{N}:0<q<b}q\left\vert\left\vert \frac{qa}{b}\right\vert\right\vert,$$ and hence $$\mathcal{Z}=\left\{\min_{q\in\mathbb{N},0<q<b}q\left\vert\left\vert \frac{qa}{b}\right\vert\right\vert,\, \frac{a}{b}\in\Q\right\}.$$ Now, by a theorem of Perron (\cite{perron}), we have that $$\min_{q\in\mathbb{N}:0<q<b}q\left\vert\left\vert \frac{qa}{b}\right\vert\right\vert=\frac{1}{\max_{1\leq i\leq n}([a_i,a_{i+1},...,a_n]+[0;a_{i-1},a_{i-2},...,a_1])},$$ where $\frac{a}{b}=[a_0;a_1,a_2,...,a_n]$ is its continued fraction expansion.
\section{Early considerations of admissible configurations}\label{Early considerations of admissible configurations}
Let $\rho$ denote some rational number with continued fraction expansion $[a_0;a_1,a_2,...,a_n]$. The chain $T=[a_1a_2...a_n]$ will be the chain associated to $\rho$ and conversely we associate to a chain $[a_1a_2,...a_n]$ the rational number equal to $\rho_T=[0;a_1,...,a_n]$. We shall be using $T$ and $\rho$ interchangably. Since $$\mathcal{Z}^{-1}(\rho)=\max_{i\in\{1,...,n\}}\left([a_i;a_{i+1},...,a_n]+[0;a_{i-1},...,a_1]\right),$$ if $T$ contains some number larger than $2$, we clearly have $\mathcal{Z}^{-1}(\rho)\geq3.$ We call the rational numbers $\rho$ satisfying $\mathcal{Z}(\rho)>\frac{1}{3}$ and their corresponding chains admissible. Hence, from now on we assume that the continued fraction sequence of $\rho$ is composed only of $1$'s and $2$'s. Markoff proved his celebrated theorem by first analyzing the properties of admissible chains. The main difference from our point of view is that one has to consider finite chains as opposed to the Markoff spectrum. However, more or less the following lemmata or variants of theirs can be found in textbooks with a standard treatment of the Markoff spectrum above $\frac{1}{3}$. See for example \cite{cassels}, \cite{MarkoffandLagrange}, \cite{bombieri2007continued}. We include their proofs for the convenience of the reader.

\begin{lemma} \label{121}
    If $T$ contains the chain $[121]$, then $\mathcal{Z}(T)\leq \frac{1}{3}.$
\end{lemma}
\begin{proof}
    If $T$ contains the chain $[121]$, then we have be definition $$\mathcal{Z}^{-1}(T)\geq [2;1,1]+[0;1,1]=3.$$
\end{proof}
Notice that a chain of the form $[a_1a_2...a_n2]$ corresponds to the same rational number as the chain $[a_1a_2,...,a_n11]$ and hence we identify these chains. Furthermore, because of the identity $\mathcal{Z}(\frac{a}{b})=\mathcal{Z}(1-\frac{a}{b}),$ we also identify the chains $[11a_1a_2,...a_n]$ and $[2a_1a_2...a_n].$ We make the following convention: \\ \\
\textit{Convention:} By the identifications above, we will assume from now on that every chain starts and ends with an even number of $2$'s. This number is allowed to be zero.
\begin{lemma} \label{212}
    If $T$ contains the chain $[212]$ and is not equal to $[11111]$, then $\mathcal{Z}(T)\leq \frac{1}{3}.$
\end{lemma}
\begin{proof}
    If $T$ is not equal to $[11111]$, then $T$ contains the chain $[212]$, which can be extended inside $T$. If it is extended by a $1$, this contradicts Lemma \ref{121}. If it is extended by a $2$, we have that $T$ contains the chain $[2212]$ or the chain $[2122]$. In either case, $$\mathcal{Z}^{-1}(T)\geq [2;1,2]+[0;2,1]=3.$$
\end{proof}
For some finite chain $T$, define $T^*$ to be its reverse chain. A section of $T$ is by definition an ordered pair of finite chains $R,S$ such that $[T]=[RS].$ We define $\mathcal{Z}(R\vert S)=\frac{1}{\rho_{R^*}+\frac{1}{\rho_{S}}}.$
\begin{lemma}\label{section 1/3}
    $\mathcal{Z}(S^*11\vert22S)=\frac{1}{3}.$
\end{lemma}
\begin{proof}
    This is simply a restatement of the identity $$[0;2,x]+[0;1,1,x]=1.$$
\end{proof}
 By definition, $\mathcal{Z}(T)$ is equal to the minimum $\mathcal{Z}(R\vert S)$ as $(R,S)$ ranges over all of its sections. It is clear now that we have the following identity $$\mathcal{Z}(T)=\mathcal{Z}(T^*).$$ 
 \begin{defn}
 We define the following ordering between chains: $$[T_1]<[T_2]\iff \rho_{T_1}>\rho_{T_2}.$$
 \end{defn}
\begin{lemma}\label{R>S}
    If $T$ has the section $(R^*1122S)$ and $\mathcal{Z}(T)>\frac{1}{3}$ then $[S]<[R].$
\end{lemma}
\begin{proof}
    This is a direct corollary of Lemma \ref{section 1/3} combined with monotonicity properties of continued fractions.
\end{proof}

\begin{lemma}\label{only even}
    If $T=[1^{e^1_{1}}2^{e^2_{1}}1^{e^1_{2}}2^{e^2_{2}}...1^{e^1_{n}}2^{e^2_{n}}]$, where $e^j_i$ are positive integers, except possibly $e^1_1$ and $e^2_n$ which are allowed to be $0,$ then all $e^j_i$ are even or $T$ consists only of $1$'s or only of $2$'s.
\end{lemma}
\begin{proof}
    Assume that $T$ contains an odd power of $1$. We then have a section of either $T$ or $T^*$ of the form $[R^*1^l\vert 2^m 1^n S],$ where $l$ is odd, $m\geq 1$ and $n$ is allowed to be zero. If $l=1$, then either $R=\emptyset$, which forces $T$ to consist only of $1$'s, or $R\neq\emptyset$ and we have a contradiction by Lemma \ref{212}. By the convention and the fact that we cannot have a subchain of the form $[121],$ $m\geq 2.$ Hence by Lemma \ref{R>S}, we have $$[2^{m-2}1^nS]<[1^{l-2}R].$$ This forces $m$ to be equal to $2$ and $n\leq l-2$ to be odd. Considering the section $[S^*1^n\vert 2^m1^lR]$ of $T^*,$ we have that $l\leq n-2$ in a similar manner- a contradiction. Therefore, $1$ appears only in even clusters. Similar considerations regarding the clusters of $2$ conclude the proof of the lemma.
\end{proof}

From now on we assume that our chain does not consist only of $1$'s or only of $2$'s. These chains are already trivially known to be admissible. Since $1$'s and $2$'s appear in even clusters, we will from now own be denoting $a=[22]$ and $b=[11],$ following Bombieri (\cite{bombieri2007continued}). 
\begin{lemma}\label{1 or 2 is zero}
    With the notation of Lemma \ref{only even}, for either $j=1$ or $j=2$, $$e^j_i\leq 2$$ for all $i$.
\end{lemma}
\begin{proof}
Suppose that this is not the case. Then there exists at least one subchain of the form $[b^e(ab)^ka^f]$ with $e,f\geq 2.$ Consider $k$ minimal such that a subchain of that form exists. Then for some chains $R,S$ $T$ has a section of the form $[R^*b^e(ab)^{k-1}ab\vert aa^{f-1}S].$ By Lemma \ref{R>S}, $$[a^{f-1}S]<[a(ba)^{k-1}b^eR].$$ Hence, $f=2$ and $S$ must start with $(ba)^{k'}b^{s_1},$ where $k'\leq k-1$ and $s_1\geq 2$, which is absurd by minimality of $k$.
\end{proof}

\section{Nielsen moves on finite words}\label{Nielsen moves on finite words}
Regarding the start and ending of admissible chains we have the following:
\begin{lemma}\label{starting}
    If $\mathcal{Z}(T)>\frac{1}{3},$ then $T$ starts and ends with $b$.
\end{lemma}
\begin{proof}
    Assume that $T$ ends with $a$. Then $T$ has a section of the form $[R^*b^sa^kb^{m-1}b\vert aa^{l-1}],$ where $m\geq 1$ and $l\geq 2$, since if $l=1$, the chain ends in $[122],$ which by our identification is the same as $[1211]$, which is absurd by Lemma \ref{121}. By Lemma \ref{R>S} $$[a^{l-1}]<[b^{m-1}a^kb^sR]$$ and therefore we must have $m=1$, $k\leq l-1,$ $s=0,$ $R=\emptyset$ and thus $T=[a^kba^l]$ with $k<l.$ Symmetrically, we obtain $l<k$, which is a contradiction.
\end{proof}
By Lemma \ref{1 or 2 is zero} and Lemma \ref{starting}, an admissible $T$ can have one of the following two forms: \begin{enumerate}
    \item Type I: $b^{e_1}ab^{e_2}a...ab^{e_n}$ or
    \item Type II: $ba^{e_1}ba^{e_2}...a^{e_n}b$.
\end{enumerate}

In either case, we call $(e_1,e_2,...,e_n)$ the characteristic sequence of $T$. Let $A$ denote the set of admissible words. For a word on the letters $a$ and $b$, the Nielsen moves $U$ and $V$ are classically defined as: $$U:a\rightarrow ab, b\rightarrow b$$ and $$V:a\rightarrow a, b\rightarrow ab.$$ The Nielsen moves are automorphisms of the free group $F_2=\langle a,b\rangle$ and every automorphism of the free group can be constructed as a word on the autoprhisms $U,V$ and their inverses. When dealing with finite words, it will become apparent that we have to ``correct" the definition of Nielsen moves, so that we respect the admissibility property of our chains. We define the variants of $U$ and $V$ on $A$, which we shall call $\overline{U}$ and $\overline{V}$ as follows. For $T\in A$, $$\overline{U}(T)=bU(T)\text{ and }\overline{V}(T)=a^{-1}V(T).$$ Notice that $\overline{V}$ is well defined since $V(T)$ always starts with an isolated copy of $a$. In the next proposition we shall show that $\overline{U}$ and $\overline{V}$ fix $A$ and furthermore, $\overline{U}$ is invertible on words of Type II and $\overline{V}$ is invertible on words of Type I. We will be referring to these adjusted Nielsen moves as $A-$Nielsen moves. We prove:

\begin{prop}\label{Nielsen moves}
    $$\mathcal{Z}(T)>\frac{1}{3}\iff \mathcal{Z}(\overline{U}(T))>\frac{1}{3}$$ and similarly $$\mathcal{Z}(T)>\frac{1}{3}\iff \mathcal{Z}(\overline{V}(T))>\frac{1}{3}.$$
\end{prop}
\begin{proof}
    $T$ is either of Type I or II. Suppose that $T$ is Type II, since the proof is identical in both cases. We can hence write $T=[ba^{e_1}...a^{e_n}b].$ By Lemma \ref{R>S} and standard monotonicity properties of continued fractions, $T$ is admissible if and only if $$(e_i-1,e_{i+1},...,e_n,\infty)<(e_{i-1},e_{i-2},...,e_1,\infty)$$ for $i=2,3,4,...,n$ in the lexicographic ordering and the same holds for $T^*.$ Notice that the move $\overline{V}$ simply increases all the $e_i$'s by $1$ and hence the equivalance is clear. As for the move $\overline{U},$ $\overline{U}(T)$ will be Type I and by Lemma \ref{R>S}, $\overline{U}(T)$ is admissible if and only if $$(e'_i-1,e'_{i+1},...,e'_k,1,1,1,...)<(e'_{i-1},e'_{i-2},...,e'_1,1,1,1,...),$$ where $\{e'_i\}_{i=1}^k$ is the characteristic sequence of $\overline{U}(T)$ and the same holds for $\overline{U}(T)^*.$ However, notice that the characteristic sequence of  $\overline{U}(T)$ is equal to $$(2,\underbrace{1,1,...,1}_{e_1-1},2,\underbrace{1,1,...,1}_{e_2-1},...,\underbrace{1,1,...,1}_{e_n-1},2),$$ from which the equivalence follows. Working similarly with $\overline{U}(T)^*$ concludes the proof of the lemma.
\end{proof}
The following theorem classifies all admissible chains.
\begin{thm}
Let $T$ some chain that does not consist only of $1$'s or only of $2$'s. Then, $T$ is admissible if and only if there exists a finite word $\overline{\Psi}$ in the $A-$Nielsen moves $\overline{U}$ and $\overline{V}$ and some $k\in\mathbb{N}$ such that $$T=\overline{\Psi}(b(ab)^k).$$
\end{thm}
\begin{proof}
    The reverse dirertion is clear by Proposition \ref{Nielsen moves}  since $T=b(ab)^k$ is admissible, as one can deduce by Lemma \ref{R>S}. For the other direction, we induct on the length  $N(T)$ of the word $T$ in the letters $a$ and $b$. If $N(T)=3$, then we must have $T=bab.$ Notice now that $N(\overline{U}(T))>N(T)$ and  $N(\overline{V}(T))>N(T)$. Hence, according to whether $T$ is Type I or II, we can apply $\overline{U}^{-1}$ or  $\overline{V}^{-1}$ accordingly and decrease the length of the word. Say $T$ is of Type I. By induction, $\overline{V}^{-1}(T)=\overline{\Psi}((b(ab)^k))$ for some word $\overline{\Psi}$ on $\overline{U}$ and $\overline{V}.$ Defining $\overline{\Psi}'=\overline{U}\cdot\overline{\Psi}$ concludes the proof of the theorem.
\end{proof}
After classifying all possible admissible chains, we wish to understand their corresponding $\mathcal{Z}$ function, or equivalently their minimal splitting.
\begin{lemma} \label{w}
For a word of Nielsen moves $\Psi$, we have that $$\Psi([ab])=awb,$$ where $w$ is symmetric. 
\end{lemma}
\begin{proof}
    Fix $k.$ We induct on the length of $\Psi.$ If $N(\Psi)=0,$ the result clearly holds. Assume, it holds for $\Psi.$ Then $U(\Psi([ab]))=U(awb)=abU(w)b$. However, $$bU(w)=bU(w^*)=bb^{-1}(U(w))^*b=(bU(w))^*,$$ and hence $bU(w)$ is symmetric. With similar cosiderations for $V\Psi$, we conclude the result.
\end{proof}

The following proposition will essentially give us our minimal splitting for $\overline{\Psi}(b(ab)^k)$.

\begin{prop}
    Let $\Psi$ a word on the Nielsen moves $U$,$V$ and let $\overline{\Psi}$ denote the corresponding word on the $A-$Nielsen moves. Then, there is a unique section $\Pi_1\Pi_2$ of $\Psi(ab)$ such that $\Pi_1,\Pi_2$  are symmetric. We then have $$\overline{\Psi}(b(ab)^k)=[\Pi_2\underbrace{\Pi_1\Pi_2...\Pi_1\Pi_2}_{k}]$$ 
\end{prop}
\begin{proof}
    Fix $k$. We use induction on the length of the word $\Psi$, or equivalently the length of the word $\overline{\Psi}.$ If $N(\Psi)=0$, then $\Pi_1=a$ and $\Pi_2=b$. Now suppose the theorem holds for all $\Psi$ with $N(\Psi)<t.$ Let $\Psi$ have length $t$. Suppose the first letter of $\Psi$ is $U$ without loss of generality. Then $U^{-1}\Psi$ has length $t-1,$ and therefore we can uniquely write $U^{-1}\Psi(ab)=\Pi_1\Pi_2$ and $\overline{U^{-1}}\overline{\Psi}(b(ab)^k)=[\Pi_2\underbrace{\Pi_1\Pi_2...\Pi_1\Pi_2}_{k}]$. Define $\Pi'_1=U(\Pi_1)b^{-1}$ and $\Pi'_2=bU(\Pi_2)$ and notice that the section $\Pi'_1\Pi'_2$ satisfies the wanted properties. As for uniqueness, assume that we can write $\Psi(ab)=\Pi''_1\Pi''_2,$ where \begin{enumerate}
        \item $\Pi''_1$ is symmetric and starts and ends with $a$
        \item $\Pi''_2$ is symmetric and starts and ends with $b$.
    \end{enumerate}  Then $$U^{-1}\Psi(ab)=U^{-1}(\Pi''_1b)U^{-1}(b^{-1}\Pi''_2).$$  By induction, $$U^{-1}(\Pi''_1b)=U^{-1}(\Pi'_1b)\text{ and }U^{-1}(b^{-1}\Pi''_2)=U^{-1}(b^{-1}\Pi'_2),$$ from which uniqueness follows.
\end{proof}

We shall refer to this unique section of $\Psi(ab)$ as the symmetric section. We have proved that every admissible word has the form $[\Pi_2\underbrace{\Pi_1\Pi_2...\Pi_1\Pi_2}_{k}],$ where $\Pi_1\Pi_2$ is the symmetric section of $\Psi(ab)$. The next lemma will be our main ingredient for finding the minimal section of $\overline{\Psi}(b(ab)^k)$.

\begin{lemma} \label{ordering}
    Let $\Psi$ some finite Nielsen word. Write $$\Psi(ab)=a^{e_1}ba^{e_2}...a^{e_n}b,$$ if $\alpha$ is of Type I and $$\Psi(ab)=ab^{e_1}ab^{e_2}...ab^{e_n},$$ if $\alpha$ is of Type II. \\ \\ We then have $$(e_1,e_2,e_3,...,e_n,\infty)>(e_i,e_{i+1},...,e_{n},\infty)$$ and $$(e_n,e_{n-1},e_{n-2},...,e_1,\infty)<(e_{n-i+1},e_{n-i},...,e_{1},\infty),$$ for all $i\in\{2,...,n\}$, if $\alpha$ is of Type I and $$(e_1,e_2,e_3,...,e_n,1,1,...)<(e_i,e_{i+1},...,e_{n},1,1,..),$$ and $$(e_n,e_{n-1},e_{n-2},...,e_1,1,1,..)>(e_{n-i+1},e_{n-i},...,e_{1},1,1,..),$$ for all $i\in\{2,...,n\}$, if $\alpha$ is of Type II. The ordering here is lexicographic.
    
\end{lemma}

\begin{proof}
The assertion is clear for the empty Nielsen word. Assume it is true for all Nielsen words up to length $N$. Then for some word of length $N$, say of Type I, of the form $a^{e_1}ba^{e_2}...a^{e_n}b$, we have that $$V(a^{e_1}ba^{e_2}...a^{e_n}b)=a^{e_1+1}ba^{e_2+1}...a^{e_n+1}b,$$ and the conclusion follows. On the other hand, the word $U(a^{e_1}ba^{e_2}...a^{e_n}b)$ is now of Type II with sequence of exponents equal to $$(\underbrace{1,1,...,1}_{e_1-1},2,\underbrace{1,1,...,1}_{e_2-1},...,\underbrace{1,1,...,1}_{e_n-1},2).$$ 
\end{proof}

\begin{thm}\label{Thm unique symmetric}
   Let $\Psi$ some Nielsen word in $a=(2,2)$ and $b=(1,1)$. For some positive integer $k$, let $r^{\Psi}_k$ the rational number with corresponding chain $[\Pi_2\underbrace{\Pi_1\Pi_2...\Pi_1\Pi_2}_{k}],$ where $\Pi=\Pi_1\Pi_2$ is the unique symmetric section of $\Psi(ab).$ Then 
   $$\mathcal{Z}(r^{\Psi}_k)=\left([\Pi]+\frac{1}{[\Pi_2\underbrace{\Pi\Pi\cdot\cdot\cdot\Pi}_{k-1\text{ times}}]}\right)^{-1}.$$

    \end{thm}
\begin{proof}
     Let $Q\vert P$ denote a section of $r^{\Psi}_k$ such that $\mathcal{Z}(r^{\Psi}_k)=\mathcal{Z}(Q\vert P).$ By Lemma \ref{ordering}, we know that $Q=\Pi_2\underbrace{\Pi\Pi\cdot\cdot\cdot\Pi}_{s\text{ times}}$ and  $P=\underbrace{\Pi\Pi\cdot\cdot\cdot\Pi}_{k-s\text{ times}}.$ However, since $\Pi$ has even length in $1$ and $2$ we get by standard continued fraction theory that since $P$ and $Q$ are even, $\rho_P+\frac{1}{\rho_Q}$ is maximized for $s=k-1$ and therefore $$\mathcal{Z}(r^{\Psi}_k)=\left([\Pi]+\frac{1}{[\Pi_2\underbrace{\Pi\Pi\cdot\cdot\cdot\Pi}_{k-1\text{ times}}]}\right)^{-1}.$$
\end{proof}

By putting together the points constructed in Theorem \ref{Thm unique symmetric}, the chains that consist only of 1's and the chains that consist only of 2's, we have  found all the points inside $\mathcal{Z}\cap (\frac{1}{3},\infty):$

\begin{corollary}\label{main cor}
   $\mathcal{Z}\bigcap (\frac{1}{3},\infty]$ is equal to $$\left\{\mathcal{Z}(r^{\Psi}_k),\text{$\Psi$ is a Nielsen word and $k\in\mathbb{N}$}\right\}\cup\left\{\frac{F_{n-2}}{F_n},n\geq 2\right\}\cup\left\{\left(\frac{1}{2}+[2;\underbrace{2,2,...,2}_{n}]\right)^{-1},n\geq 0\right\}.$$ 
\end{corollary}

\section{Limit points of $\mathcal{Z}\cap (\frac{1}{3},\infty)$ and limit points of $\mathcal{Z}\cap (\frac{1}{3}-\epsilon,\frac{1}{3})$.}\label{Limit points}

Notice that if $k_i$ is some sequence of positive integers and $\Psi_i$ is a sequence of Nielsen words on $U$ and $V$ satisfying $N(\Psi_i)\rightarrow\infty$, then by Lemma \ref{section 1/3} $$\mathcal{Z}(r_{k_i}^{\Psi_i})\rightarrow\frac{1}{3}.$$ Hence, in order to study the limit points of $\mathcal{Z}$ larger than $\frac{1}{3}$, we have to understand for a fixed $\Psi$ the limits $\lim_{n\rightarrow\infty}\mathcal{Z}(r_n^{\Psi}).$  \\ \\ 
To a continued fraction of a rational number $\rho$ of the form $[0;a_1,a_2,...,a_n]$ with $n$ even, we associate the matrix $$M_{\rho}=\begin{pmatrix}
a_1 & 1\\
1 & 0 
\end{pmatrix}\cdot\begin{pmatrix}
a_2 & 1\\
1 & 0 
\end{pmatrix}\cdot\cdot\cdot\begin{pmatrix}
a_n & 1\\
1 & 0 
\end{pmatrix}.$$
Markoff showed that the triple$$(\frac{Tr(\Psi(a))}{3},\frac{Tr(\Psi(b))}{3},\frac{Tr(\Psi(ab))}{3})$$ is a solution $(x,y,z)$ of the Markoff equation, with no two entries sharing a common prime factor. He further showed, that $M_{\Psi(ab)}$ has the following form: 
$$M_{\Psi(ab)}=\begin{pmatrix}
3z-z' & *\\
z & z' 
\end{pmatrix},$$ where $z'$ is the unique integer solution of the equation $$x+yz'\equiv 0\mod z\text{ and }0<z'<z.$$

\begin{proof}[Proof of Theorem \ref{Limit points}]
By Theorem \ref{Thm unique symmetric} and the argument in the beginning of the section, all the limit points of $\mathcal{Z}$ larger than $\frac{1}{3}$ are of the form: 
\begin{enumerate}
    \item $\frac{1}{(\frac{1+\sqrt{5}}{2})^2},$
    \item $\frac{2}{(1+\sqrt{2})^2},$
    \item $\frac{1}{[\Pi_2\overline{\Pi}]^{-1}+[\Pi]},$ where here $\overline{\Pi}$ denotes the periodization of $\Pi$ and $\Pi=\Pi_1\Pi_2$ denotes the symmetric splitting of $\Psi(ab)$ for some Nielsen word $\Psi.$ 
\end{enumerate}

We know by definition of the matrix $M_{\Psi(ab)}$ that $[\Pi]=\frac{3z-z'}{z}$ and that $[\Pi_2\overline{\Pi}]^{-1}$ is equal to  $$\frac{(3z-z')-z'-\sqrt{9z^2-4}}{2z}.$$ Therefore, $$\frac{1}{[\Pi_2\overline{\Pi}]^{-1}+[\Pi]}=\frac{1}{-\frac{(3z-2z')-\sqrt{9z^2-4}}{2z}+\frac{3z-z'}{z}}=\frac{1}{\frac{\sqrt{9z^2-4}}{2z}+\frac{3}{2}}.$$
\end{proof}

We now prove Theorem \ref{Uncountably}, showing that below the point $\frac{1}{3}$, we can find uncountably many limit points as close as we desire:

\begin{proof}[Proof of Theorem \ref{Uncountably}]:
    For any finite sequence of positive integers $n_1,n_2,...,n_k$ define $$R(n_1,n_2,...,n_k)=[0;\underbrace{1,1,...,1}_{n_1},2,2,\underbrace{1,1,...,1}_{n_2},2,2,...,2,2,\underbrace{1,1,...,1}_{n_k}]$$
    Clearly by Lemma \ref{section 1/3}, there exists some $N_0$, such that for any $n_1,n_2,...,n_k\geq N_0:$ $$\mathcal{Z}\left(R(n_1,n_2,...,n_k)\right)>\frac{1}{3}-\epsilon.$$ Choose now $n_1$ to be odd. By the description of all the points in the spectrum above $\frac{1}{3},$ we know that $\mathcal{Z}\left(R(n_1,n_2,...,n_k)\right)\leq\frac{1}{3}$ and by picking $n_1< n_i$ for all $i\in\{2,3,...,k\}$ we can make the inequality strict. We conclude that $$\mathcal{Z}\left(R(n_1,n_2,...,n_k)\right)\in(\frac{1}{3}-\epsilon,\frac{1}{3}).$$

    Choose $n_2$ even such that $n_i<n_2$ for all $i\in\{1,3,...,k\}$. Then, we can see the maximal section is  $$[\underbrace{1,1,...,1}_{n_1}\vert2,2,\underbrace{1,1,...,1}_{n_2},2,2,...,2,2,\underbrace{1,1,...,1}_{n_k}],$$ and $$\mathcal{Z}\left(R(n_1,n_2,...,n_k)\right)=\frac{1}{[2;2,\underbrace{1,1,...,1}_{n_2},2,2,...,2,2,\underbrace{1,1,...,1}_{n_k}]+[0;\underbrace{1,1,...,1}_{n_1}]}.$$ It is clear now that by choosing $(n_3,n_4,...,n_k)$  in the range $(n_1,n_2)$ and taking $k\rightarrow\infty$, we can construct uncountably many limit points of the form $$\lim_{k\rightarrow\infty}\frac{1}{[2;2,\underbrace{1,1,...,1}_{n_2},2,2,...,2,2,\underbrace{1,1,...,1}_{n_k},...]+[0;\underbrace{1,1,...,1}_{n_1}]}.$$
\end{proof}

\section*{Acknowledgements}
I wish to thank my advisor Peter Sarnak for suggesting this problem to me and for many useful remarks on earlier drafts of this paper.

\bibliographystyle{abbrv}
\bibliography{biblio.bib}

\end{document}